\documentclass[12pt,a4paper]{article}
\usepackage{amssymb} 
\usepackage{amsmath} 
\usepackage{amsthm} 
\usepackage{hyperref}
\usepackage{mathrsfs} 
\usepackage{a4wide}

\newtheorem{theorem}{Theorem}
\newtheorem{problem}[theorem]{Problem}
\newtheorem{proposition}[theorem]{Proposition}
\newtheorem{lemma}[theorem]{Lemma}
\newtheorem{corollary}[theorem]{Corollary}
\newtheorem{conjecture}[theorem]{Conjecture}

\newtheorem*{slll}{The Lov\'asz Local Lemma}
\newtheorem*{claim}{Claim}

\renewcommand{\Pr}{\,\mathbb{P}}
\newcommand{\eps}{\varepsilon}

\newcommand*{\myproofname}{Proof}
\newenvironment{claimproof}[1][\myproofname]{\begin{proof}[#1]}{\end{proof}}

\title{Independent transversals in bipartite correspondence-covers}

\author{
Stijn Cambie
\thanks{Department of Mathematics, Radboud University, Postbus 9010, 6500 GL Nijmegen, Netherlands. 
Email: \protect\href{mailto:stijn.cambie@hotmail.com}{\protect\nolinkurl{stijn.cambie@hotmail.com}}, \protect\href{mailto:ross.kang@gmail.com}{\protect\nolinkurl{ross.kang@gmail.com}}. Supported by a Vidi grant (639.032.614) of the Dutch Research Council (NWO).}
\and
Ross J. Kang
\footnotemark[2]
}

\begin{document}

\maketitle

\begin{abstract}
Suppose $G$ and $H$ are bipartite graphs and $L: V(G)\to 2^{V(H)}$ induces a partition of $V(H)$ such that
the subgraph of $H$ induced between $L(v)$ and $L(v')$ is a matching whenever $vv'\in E(G)$.
We show for each $\varepsilon>0$ that, if $H$ has maximum degree $D$ and $|L(v)| \ge (1+\varepsilon)D/\log D$ for all $v\in V(G)$, then $H$ admits an independent transversal with respect to $L$, provided $D$ is sufficiently large. This bound on the part sizes is asymptotically sharp up to a factor $2$.
We also show some asymmetric variants of this result.
\end{abstract}

\section{Introduction}\label{sec:intro}

This note focuses on the progression from list colourings towards independent transversals in vertex-partitioned graphs, specifically for bipartite graphs. This follows close on the heels of earlier work of the authors together with Alon~\cite{ACK20+}, but since the setup is considerably strengthened we provide these results separately both for clarity and for the benefit of the interested reader.

Allow us to deliberately present list colouring of graphs in an awkward way. Let $G$ be a simple undirected graph. From a {\em list-assignment} $L$ of $G$, i.e.~a mapping $L:V(G)\to 2^{{\mathbb Z}^+}$, we derive the {\em list-cover} $H_\ell(G,L)$ for $G$ via $L$ as follows.
For every $v\in V(G)$, we let $L_\ell(v)=\{(v,c)\}_{c\in L(v)}$ and define $V(H_\ell)=\bigcup_{v\in V(G)}L_\ell(v)$. We define $E(H_\ell)$ by including $(v,c)(v',c')\in E(H_\ell)$ if and only if $vv'\in E(G)$ and $c=c'\in L(v)\cap L(v')$.
Note $L_\ell$ induces a partition of the vertices of $H_\ell$. We seek an {\em independent transversal} of $H_\ell$ with respect to this partition, i.e.~a vertex subset with exactly one vertex chosen from each part that simultaneously forms an independent set. The independent transversals of $H_\ell$ with respect to $L_\ell$ are in one-to-one correspondence with the proper $L$-colourings of $G$, as originally introduced in~\cite{Viz76,ERT80}. We remark that finding independent transversals of a general graph $H$ with respect to some partition $L$ of its vertices is another classic combinatorial problem~\cite{BES75}. In both settings, we usually seek lower bound conditions on the size of the parts in terms of the maximum degree of $H_\ell$ or $H$ that suffice for the existence of an independent transversal.

In this note, we restrict our attention almost exclusively to the case of bipartite $G$ and $H$. For this we find it helpful to introduce some finer notation. Let $G$ and $H$ be bipartite graphs with bipartitions $(A_G,B_G)$ and $(A_H,B_H)$, respectively. 
We say that $H$ is a {\em bipartite cover} of $G$ with respect to a mapping $L: V(G) \to 2^{V(H)}$ if
\begin{itemize}
\item
$H$ is a cover of $G$ with respect to $L$, i.e.~$L$ induces a partition of $V(H)$ and the subgraph induced between $L(v)$ and $L(v')$ is empty whenever $vv'\notin E(G)$; and
\item
the partition induced by $L$ agrees with the bipartitions $(A_G,B_G)$ and $(A_H,B_H)$, i.e.~$L(A_G)$ induces a partition of $A_H$ and $L(B_G)$ induces a partition of $B_H$.
\end{itemize}
The general problem here is as follows.

\begin{problem}\label{prob:general}
Let $G$ and $H$ be bipartite graphs with bipartitions $(A_G,B_G)$, $(A_H,B_H)$, respectively, 
such that $H$ is a bipartite cover of $G$ with respect to some $L: A_G\to 2^{A_H}, B_G\to 2^{B_H}$.
What conditions on $G$, $H$, and integers $k_A$, $k_B$, $\Delta_A$, $\Delta_B$, $D_A$, $D_B$ (where $\Delta_A$, $\Delta_B$ are possibly $\infty$) suffice to ensure the following? If the maximum degrees in $A_G$, $B_G$, $A_H$, $B_H$ are $\Delta_A$, $\Delta_B$, $D_A$, $D_B$, respectively, and $|L(v)|\ge k_A$ for all $v\in A_G$ and $|L(w)|\ge k_B$ for all $w\in B_G$, then there is guaranteed to be an independent transversal of $H$ with respect to $L$.
\end{problem}

\noindent
Although our considerations are broader, the symmetric version of Problem~\ref{prob:general} with $\Delta_A=\Delta_B=\Delta$ (possibly infinite) and $D_A=D_B=D$ is perhaps most natural. We note in this case that without further conditions on $G$ and $H$, Problem~\ref{prob:general} is already close to settled. This is due to a seminal result of Haxell~\cite{Hax01}, which implies that $k_A = k_B=2 D$ suffices. Furthermore, this is not far from sharp by considering $H$ to be a complete bipartite graph with $D$ vertices in each part (and $G$ an independent edge).
We next consider what happens when we impose some mild structural constraints on $G$ and $H$.

For $H$ assumed to be a (bipartite) list-cover of $G$ with respect to $L$, we studied Problem~\ref{prob:general} in some depth in our previous work with Alon~\cite{ACK20+}. For this form, it was conjectured in 1998 by Alon and Krivelevich~\cite{AlKr98} that if $\Delta_A=\Delta_B=\Delta$ (and vacuously $D_A=D_B=\Delta$), then for some absolute constant $C>0$, $k_A=k_B=C\log \Delta$ suffices. It was shown in this same setting in~\cite{ACK20+} that for each $\eps>0$, $k_A=\log \Delta$ and $k_B=(1+\eps)\Delta/\log\Delta$ suffices provided $\Delta$ is large enough.

Here we relax the setting somewhat by considering the consequences of our previous findings for a natural generalisation of list-covers.
A {\em correspondence-assignment} for $G$ via $L$ is a cover $H$ for $G$ via $L$ such that for each edge $vv'\in E(G)$, the subgraph induced between $L(v)$ and $L(v')$ is a matching. We call such an $H$ a {\em correspondence-cover} for $G$ with respect to $L$.
Every list-cover is a correspondence-cover.

Consider Problem~\ref{prob:general} assuming that $H$ is a correspondence-cover of $G$ with respect to $L$.
This leads to two natural forms of Problem~\ref{prob:general} that are strengthenings of the one that we studied with Alon in~\cite{ACK20+}.
If we only have conditions on $\Delta_A$ and $\Delta_B$ (so with $D_A\ge \Delta_A$, $D_B\ge \Delta_B$), then this is related to {\em correspondence colouring}~\cite{DvPo18,FHK16}.
If instead we only have conditions on $D_A$ and $D_B$ (so with $\Delta_A=\Delta_B=\infty$), then this is related to finding independent transversals with respect to partitions of local degree $1$, as originally proposed by Aharoni and Holzman, see~\cite{LoSu07}. Both of these stronger variants of list colouring have seen interesting recent advances, see~e.g.~\cite{Ber19,Mol18+,DKPS20+,GlSu20+,KaKe20+}.
Although here it is corollary to the work of Loh and Sudakov~\cite{LoSu07} that for each $\eps>0$, $k_A = k_B=(1+\eps)D$ suffices provided $\Delta$ is large enough, we will see how to improve on this statement in a way that is nearly optimal in various regimes.

The purpose of this note is to show the following progress towards Problem~\ref{prob:general} specific to correspondence-covers.

\begin{theorem}\label{thm:main}
Let $G$ and $H$ be bipartite graphs with bipartitions $(A_G,B_G)$, $(A_H,B_H)$, respectively, 
such that $H$ is a bipartite correspondence-cover of $G$ with respect to some $L: A_G\to 2^{A_H}, B_G\to 2^{B_H}$.
Assume one of the following conditions, as stated or with roles exchanged between $A$ and $B$.
\begin{enumerate}
\item\label{itm:transversals}
$k_B \ge (ek_A D_B)^{1/(k_A-1)}D_A$.
\item\label{itm:coupon}
\(
e(k_AD_A(k_BD_B-1)+1) \left(1-(1-1/k_B)^{D_A}\right)^{k_A} \le 1.
\)
\item\label{itm:couponDP}
\(
e(\Delta_A(\Delta_B-1)+1) \left(1-(1-1/k_B)^{\Delta_A\min\left\{1,k_B/k_A\right\}}\right)^{k_A} \le 1.
\)
\end{enumerate}
If the maximum degrees in $A_G$, $B_G$, $A_H$, $B_H$ are $\Delta_A$, $\Delta_B$, $D_A$, $D_B$, respectively, and $|L(v)|\ge k_A$ for all $v\in A_G$ and $|L(w)|\ge k_B$ for all $w\in B_G$, then $H$ admits an independent transversal with respect to $L$.
\end{theorem}

This result is quite similar to one in earlier work~\cite[Thm.~4]{ACK20+} and uses the same methods, but here the setting is considerably stronger. To illustrate, we next indicate how each of the three conditions in Theorem~\ref{thm:main} is close to sharp in certain regions. We do not have the same sharpness in the list-cover case, and so further progress for list-covers (and, hopefully, in the conjecture of Alon and Krivelevich~\cite{AlKr98}) will have to take advantage of some special structure not necessarily present in correspondence-covers.

We highlight and discuss three corollaries of  Theorem~\ref{thm:main}. 

First from condition~\ref{itm:coupon} of Theorem~\ref{thm:main} we may conclude the following symmetric result. 

\begin{corollary}\label{cor:symmetric}
For each $\eps>0$, the following holds for $D_0$ sufficiently large.
Let $G$ and $H$ be bipartite graphs such that $H$ is a correspondence-cover of $G$ with respect to some $L: V(G) \to 2^{V(H)}$.
If $H$ has maximum degree $D\ge D_0$ and $|L(v)| \ge (1+\eps)D/\log D$ for all $v\in V(G)$, 
then $H$ admits an independent transversal with respect to $L$.
\end{corollary}

\noindent
For an appreciation of the strength of this result, let us note that the part size bound in Corollary~\ref{cor:symmetric} is sharp up to an asymptotic factor $2$, see~\cite[Thm.~1]{KPV05}.

Note that an immediate consequence of Corollary~\ref{cor:symmetric} is that if we assume (moreover) that the {\em covered} graph $G$ has maximum degree $\Delta \le D$, then the same conclusion holds. This is equivalent to correspondence colouring of bipartite graphs, and so this weaker assertion also follows from a recent result of Bernshteyn~\cite{Ber19} (see also~\cite{DJKP20}) on correspondence colouring of triangle-free graphs.
On the other hand, if one analogously relaxes the conditions on $G$ and $H$ in Corollary~\ref{cor:symmetric}, i.e.~suppose instead that $G$ and $H$ are triangle-free, then it is unknown whether or not a part size bound that is $o(D)$ as $D\to\infty$ suffices. It could even be possible for the following to be true.

\begin{conjecture}\label{conj:trianglefree}
For each $\eps>0$, the following holds for $D_0$ sufficiently large.
Let $G$ and $H$ be graphs such that $H$ is a correspondence-cover of $G$ with respect to some $L: V(G) \to 2^{V(H)}$.
If $G$ is triangle-free, $H$ has maximum degree $D\ge D_0$, and $|L(v)| \ge (1+\eps)D/\log D$ for all $v\in V(G)$, 
then $H$ admits an independent transversal with respect to $L$.
\end{conjecture}

\noindent
It is worth mentioning that there is independent supporting evidence towards Conjecture~\ref{conj:trianglefree}.  Specifically Amini and Reed~\cite{AmRe08} and Alon and Assadi~\cite{AlAs20+} independently obtained a particular list-colouring result for triangle-free graphs, a result which is only slightly weaker than the statement of Conjecture~\ref{conj:trianglefree}.
If true, Conjecture~\ref{conj:trianglefree} would directly extend along an important line of work going back to the seminal results of Ajtai, Koml\'os, Szemer\'edi~\cite{AKS81} and Johansson~\cite{Joh96}.
If one were bolder, one could also posit Conjecture~\ref{conj:trianglefree} holding under further relaxed conditions, namely, that $G$ is complete and $H$ is triangle-free.

Second from condition~\ref{itm:couponDP} of Theorem~\ref{thm:main} the next asymmetric result follows easily. This is a modest generalisation of an earlier result for list-covers~\cite[Cor.~10]{ACK20+}.

\begin{corollary}\label{cor:symmetricasymmetric}
For each $\eps>0$, the following holds for $\Delta_0$ sufficiently large.
Let $G$ and $H$ be bipartite graphs with bipartitions $(A_G,B_G)$, $(A_H,B_H)$, respectively, 
such that $H$ is a bipartite correspondence-cover of $G$ with respect to some $L: A_G\to 2^{A_H}, B_G\to 2^{B_H}$.
If $G$ has maximum degree $\Delta \ge \Delta_0$, $|L(v)| \ge (1+\eps)\Delta/\log_4\Delta$ for all $v\in A_G$, and $|L(w)|=2$ for all $w\in B_G$, 
then $H$ admits an independent transversal with respect to $L$.
\end{corollary}

\noindent
Our reason for highlighting this bound in particular though is that the part size bound in Corollary~\ref{cor:symmetricasymmetric} is asymptotically sharp, as certified by the following construction.

\begin{proposition}\label{prop:symmetricasymmetric}
For infinitely many $\Delta$, there exist bipartite graphs $G$ and $H$ with bipartitions $(A_G,B_G)$, $(A_H,B_H)$, respectively, 
such that $H$ is a bipartite correspondence-cover of $G$ with respect to some $L: A_G\to 2^{A_H}, B_G\to 2^{B_H}$ and such that the following holds.
The maximum degree of $G$ is $\Delta$, $|L(v)| = \Delta/\log_4\Delta$ for all $v\in A_G$, $|L(w)|=2$ for all $w\in B_G$, and
$H$ does not admit an independent transversal with respect to $L$.
\end{proposition}

\noindent
In the special case of $H$ a list-cover, neither a tightness result analogous to Proposition~\ref{prop:symmetricasymmetric} nor a stronger form of Corollary~\ref{cor:symmetricasymmetric} is known to hold.

As with Corollary~\ref{cor:symmetric}, one might wonder whether  Corollary~\ref{cor:symmetricasymmetric} could be strengthened to hold in the more general situation that we bound instead the maximum degree of the correspondence-cover $H$, say, by $D$.
A construction similar to that used in Proposition~\ref{prop:symmetricasymmetric} shows that this is impossible, and in fact far from possible in that if the $B$-parts are held to size $2$, then the $A$-parts cannot be size $o(D^{8/5})$; see Proposition~\ref{prop:symmetricasymmetricD} below.
One might also wonder what happens when we further relax (at least in part) the condition that $H$ be a correspondence-cover.
In Section~\ref{sec:asymHax} we observe how an aforementioned theorem of Haxell~\cite{Hax01} applies to yield the following result.

\begin{proposition}\label{prop:asymHax}
	Let $G$ and $H$ be bipartite graphs $G$ and $H$ with bipartitions $(A_G,B_G)$, $(A_H,B_H)$, respectively, 
	such that $H$ is a bipartite cover of $G$ with respect to some $L: A_G\to 2^{A_H}, B_G\to 2^{B_H}$ and such that the following holds.
	The maximum degree of $H$ is $D$, $|L(v)| \ge 2D^2$ for all $v\in A_G$, $|L(w)|=2$ for all $w\in B_G$, and no vertex of $A_H$ is adjacent to both vertices of $L(w)$ for some $w\in B_G$. Then $H$ has an independent transversal with respect to $L$.
	Moreover, the conclusion may fail if the part size condition $2D^2$ is replaced by $D^2$.
\end{proposition}

\noindent
It would be interesting to narrow the gap between $D^2$ and $2D^2$ in the above result. Similarly, in the correspondence-cover version of this problem, it would be interesting to decide on the correct asymptotic behaviour for the analogous term, which we know must lie between $\Theta(D^{8/5})$ (Proposition~\ref{prop:symmetricasymmetricD}) and $2D^2$.

Third let us consider a different asymmetric situation: suppose in Problem~\ref{prob:general} that $k_A=k_B=k$. Then from condition~\ref{itm:transversals} in Theorem~\ref{thm:main} we can read off the following.

\begin{corollary}\label{cor:asymmetric}
Let $G$ and $H$ be bipartite graphs with bipartitions $(A_G,B_G)$, $(A_H,B_H)$, respectively, 
such that $H$ is a bipartite correspondence-cover of $G$ with respect to some $L: A_G\to 2^{A_H}, B_G\to 2^{B_H}$.
If $H$ has maximum degree $1$ in part $A_H$ and maximum degree $k^{k-2}/e$ in part $B_H$, and $|L(v)| \ge k$ for all $v\in V(G)$, 
then $H$ admits an independent transversal with respect to $L$.
\end{corollary}

\noindent
Note that this statement is utterly trivial if $H$ is a list-cover, for then it boils down to colouring a forest of stars for which each leaf-vertex has more than one colour in its list.
Curiously, the statement as is for $H$ a correspondence-cover is tight up to some $O(k)$ factor.

\begin{proposition}\label{prop:asymmetric}
There exist bipartite graphs $G$ and $H$ with bipartitions $(A_G,B_G)$, $(A_H,B_H)$, respectively, 
such that $H$ is a bipartite correspondence-cover of $G$ with respect to some $L: A_G\to 2^{A_H}, B_G\to 2^{B_H}$ and such that the following holds.
The maximum degree of $H$ in part $A_H$ is $1$ and in part $B_H$ is $k^{k-1}$, $|L(v)| = k$ for all $v\in V(G)$, and
$H$ does not admit an independent transversal with respect to $L$.
\end{proposition}

Thus Corollaries~\ref{cor:symmetric},~\ref{cor:symmetricasymmetric}, and~\ref{cor:asymmetric} cannot be improved much, and so neither can conditions~\ref{itm:coupon},~\ref{itm:couponDP}, and~\ref{itm:transversals}, respectively, of Theorem~\ref{thm:main}.
Our motivation from~\cite{AlKr98,ACK20+} is Problem~\ref{prob:general} for the special case of list-covers $H$, but this note proves that further progress along these lines needs some special insight specific to list-covers but not correspondence-covers.

\subsection*{Probabilistic preliminaries}

We will use the following standard version of the local lemma.

\begin{slll}[\cite{ErLo75}]
Take a set $\cal E$ of (bad) events such that for each $A\in \cal E$
\begin{enumerate}
\item $\Pr(A) \le p < 1$, and
\item $A$ is mutually independent of a set of all but at most $d$ of the other events.
\end{enumerate}
If $ep(d+1)\le1$, then with positive probability none of the events in $\cal E$ occur.
\end{slll}

\section{Proofs}\label{sec:proofs}

Before proceeding to the main proofs, let us first show how Corollary~\ref{cor:symmetric} follows from Theorem~\ref{thm:main}.
We in fact have the following slightly more general statement. We note that this statement is reminiscent of ``local'' list-colouring results in which the list sizes can vary depending on the structural parameters of the individual vertices (such as their degree); see~\cite{DKPS20+} and especially Sec.~8.2 therein for one of the most general results along these lines.

\begin{theorem}\label{thm:local}
For each $\eps>0$, the following holds for $D_0$ sufficiently large.
Let $G$ and $H$ be bipartite graphs with bipartitions $(A_G,B_G)$, $(A_H,B_H)$, respectively, 
such that $H$ is a bipartite correspondence-cover of $G$ with respect to some $L: A_G\to 2^{A_H}, B_G\to 2^{B_H}$.
If $H$ has maximum degree $D_A$ in part $A_H$ and maximum degree $D_B$ in part $B_H$ for $D_A,D_B \ge D_0$, $|L(v)| \ge (1+\eps)D_A/\log D_A$ for all $v\in A_G$, and $|L(w)| \ge (1+\eps)D_B/\log D_B$ for all $w\in B_G$, 
then $H$ admits an independent transversal with respect to $L$.
\end{theorem}

\begin{proof}
Without loss of generality, we can assume $D_B \ge D_A.$

	If $D_B \ge  D_A^2$, say, then the result follows from condition~\ref{itm:transversals} in Theorem~\ref{thm:main}. For then, we have that \begin{align*}k_B \ge (1+\eps)D_B/\log D_B &= ((1+\eps)D_B^{1/4})(D_B^{1/4}/\log D_B)\sqrt{D_B}\\ &\ge ((1+\eps)\sqrt{D_A})(D_B^{1/4}/\log D_B)D_A.\end{align*} Now note that $$(1+\eps)\sqrt{D_A} \ge (e k_A)^{1/(k_A-1)}$$ since $$k_A \ge (1+\eps)D_A/\log D_A\mbox{, and }D_B^{1/4}/\log D_B \ge D_B^{1/(k_A-1)},$$ both provided $D_0$ is sufficiently large. So we can conclude with Theorem~\ref{thm:main} as we have verified condition~\ref{itm:transversals}.

	If $D_B \le  D_A^2$, it is a consequence of condition~\ref{itm:coupon} in Theorem~\ref{thm:main}.
	Let $\delta=\eps /2>0$. We have $1-1/k_B \ge \exp( -(1+\delta)/k_B )$ for $k_B$ sufficiently large.
	So then
	\begin{align*}
	1-(1-1/k_B)^{D_A} &\le 1- \exp \left( -(1+\delta)D_A/k_B  \right)\\
	&\le \exp \left( - \exp \left( -(1+\delta)D_A/k_B  \right)\right).
	\end{align*}
	Hence, letting $k_B \ge (1+\eps)D_B/\log D_B$, we can verify using $k_A \le D_A$ and $k_B \le D_B$ that 
	\begin{align*}
	&e(k_AD_A(k_BD_B-1)+1) \left(1-(1-1/k_B)^{D_A}\right)^{k_A} \\
	&\le e D_A^6 \exp \left( - \exp \left( -(1+\delta) D_A/k_B \right) k_A\right)\\
	&\le e D_A^6 \exp \left( - \exp \left( -((1+\delta)/(1+\eps)) \log{D_A} \right) k_A\right)\\
	&\le e D_A^6 \exp \left( - (1+\eps)   D_A^{ \delta/(1+\eps)} / \log{D_A} \right)\\
	&<1. \qedhere
	\end{align*}
\end{proof}

We now proceed to the proof of Theorem~\ref{thm:main}. Note that these arguments are nearly identical to those given in~\cite{ACK20+}, but we include them in the notation of this stronger setup and for completeness.

We apply a simple result about hypergraph transversals, which needs a little more notation. Let ${\mathscr H}=(V,E)$ be a hypergraph. The {\em degree} of a vertex in ${\mathscr H}$ is the number of edges containing it. 
Given some partition of $V$, a {\em transversal} of ${\mathscr H}$ is a subset of $V$ that intersects each part in exactly one vertex. 
A transversal of ${\mathscr H}$ is called {\em independent} if it contains no edge, see~\cite{EGL94}.
The following is in fact a modest strengthening of the main result of~\cite{ErLo75}. See~\cite[Lem.~11]{ACK20+} for its straightforward derivation with  the Lov\'asz Local Lemma.

\begin{lemma}\label{thm:transversal} Fix $k\ge 2$.
Let ${\mathscr H}$ be a $k$-uniform vertex-partitioned hypergraph, each part being of size $\ell$, such that every part has degree sum at most $\Delta$.  If $\ell^k \ge e(k(\Delta-1)+1)$, then ${\mathscr H}$ has an independent transversal.
\end{lemma}

\noindent
 Lemma~\ref{thm:transversal} implies Theorem~\ref{thm:main} under condition~\ref{itm:transversals}.

\begin{proof}[Proof of Theorem~\ref{thm:main} under condition~\ref{itm:transversals}]
Let $k_A,k_B,D_A,D_B$ satisfy condition~\ref{itm:transversals}. 
For every $w \in B_G$ with $\lvert L(w) \rvert > k_B$, take a subset of exactly $k_B$ vertices of $L(w)$ and remove the other vertices and incident edges.
We do similarly for every $v\in A_G$ with $\lvert L(v) \rvert > k_A$.
We define a suitable hypergraph ${\mathscr H}$ 
with $V({\mathscr H}) = V(H)$. 

Let $(w_1,\ldots,w_{k_A})$ be an edge of $E({\mathscr H})$ if the $w_i$ are elements from different $L(w)$, for $w \in B_G$, and there is some $v\in A_G$ such that there is a perfect matching between $\{w_1,\ldots,w_{k_A}\}$ and $L(v)$.

Note that ${\mathscr H}$ is a $k_A$-uniform vertex-partitioned hypergraph, where the parts are naturally induced by each $L(w)$, for $w \in B_G$, and so are each of size $k_B$. We have defined ${\mathscr H}$ and its partition so that any independent transversal of ${\mathscr H}$ corresponds to a partial independent transversal of $H$ with respect to $L$ that can be extended to an independent transversal of $H$.

Every vertex in ${\mathscr H}$ has degree at most
\(
D_B D_A^{k_A-1},
\)
and so the result follows from Lemma~\ref{thm:transversal} with $\ell=k_B$ and $\Delta=k_B D_B D_A^{k_A-1}$.
\end{proof}

\begin{proof}[Proof of Theorem~\ref{thm:main} under condition~\ref{itm:coupon} or~\ref{itm:couponDP}]
	Let $k_A,k_B,D_A,D_B$ satisfy condition~\ref{itm:coupon} or $k_A,k_B,\Delta_A,\Delta_B$ satisfy condition~\ref{itm:couponDP}.
	By focusing on a possible subgraph of $H$, we can assume $\lvert L(v) \rvert = k_A$ for every $v \in A_G$ and $\lvert L(w) \rvert = k_B$ for every $w \in B_G$.
	We pick randomly and independently one vertex in $L(w)$ for every $w\in B_G$, resulting in a set $\mathbf{B'}$ of $\lvert B_G \rvert$ vertices.
	Let $T_{v,c}$ be the event that for some $v \in A_G$, the vertex $c \in L(v)$ has a neighbour in $\mathbf{B'}$.
	Let $T_v$ be the event that $T_{v,c}$ happens for all $c\in L(v)$.
	
	\begin{claim}\label{claim:collector}
		The events $T_{v,c}$, for fixed $v$ as $c$ ranges over all vertices in $L(v)$, are negatively correlated. 
		In particular,
		\(\Pr(T_v) \le \prod_{c\in L(v)} \Pr(T_{v,c})\).
	\end{claim}
	
	\begin{claimproof}
		We have to prove, for every $I \subset L(v)$, that
		$\Pr(\forall c \in I \colon T_{v,c}) \le \prod_{c\in I} \Pr(T_{v,c}).$
		We prove the statement by induction on $\lvert I \rvert.$ When $\lvert I \rvert\le 1$ the statement is trivially true.
		Let $I \subset L(v)$ be a subset for which the statement is true and let $c' \in L(v) \setminus I.$ We now prove the statement for $I'=I \cup \{c'\}.$
		We have $\Pr(\forall c \in I \colon T_{v,c}) \le \Pr(\forall c \in I \colon T_{v,c} \mid \lnot T_{v,c'} )$ as the probability to forbid all vertices in $I$ is larger if no neighbour of $c'$ is selected.
		This is equivalent to 
		\begin{align*}
		\Pr(\forall c \in I \colon T_{v,c}) &\ge \Pr(\forall c \in I \colon T_{v,c} \mid T_{v,c'} )\\
		\iff\Pr(\forall c \in I' \colon T_{v,c}) &\le 	\Pr(\forall c \in I \colon T_{v,c}) \Pr(T_{v,c'})
		\end{align*}
		This last expression is at most $ \prod_{c\in I'} \Pr(T_{v,c})$ by the induction hypothesis, as desired.
	\end{claimproof}

	Let us write $L(v) = \{c_1,\dots,c_{k_A}\}$ and for each $i\in[k_A]$
	let the degree of $c_i$ in the neighbouring lists of $v$ be $x_i$.
	Note that $\Pr(T_{v,c_i}) = 1- (1-1/k_B)^{x_i}.$
	
	Under condition~\ref{itm:coupon}, we have $x_i \le D_A$ for every $i.$
	So by the claim we have 
	$$\Pr(T_v) \le \left(1- \left(1-\frac{1}{k_B}\right)^{  D_A} \right)^{k_A}.$$
	Each event $T_v$ is mutually independent of all other events $T_u$ apart from those corresponding to vertices $u\in A_G$ such that some $w_1 \in L(u)$ and $ w_2 \in L(v)$ both have a neighbour in the same part $L(b)$ for some $b \in B_G$. As there are at most $k_AD_A(k_BD_B-1)$ such vertices besides $v$, the Lov\'asz Local Lemma guarantees with positive probability that none of the events $T_v$ occur, i.e.~there is an independent transversal, as desired.

	Now we assume condition~\ref{itm:couponDP}.
	Using  $x_i \le \Delta_A$ for every $1 \le i \le k_A$ and the claim, 
	we have 
	$$\Pr(T_v) \le \left(1- \left(1-\frac{1}{k_B}\right)^{  \Delta_A} \right)^{k_A}.$$
	Noting that $\sum_{i=1}^{k_A} x_i \le k_B \Delta_A$ and that the function $\log (1- (1-1/k_B)^{x} )$ is concave and increasing, Jensen's Inequality together with the claim
	implies that 
	$$\Pr(T_v) \le \left(1- \left(1-\frac{1}{k_B}\right)^{ k_B \Delta_A/k_A} \right)^{k_A}.$$
	Each event $T_v$ is mutually independent of all other events $T_u$ apart from those corresponding to vertices $u\in A$ that have a common neighbour with $v$ in $G$. As there are at most $\Delta_A(\Delta_B-1)$ such vertices besides $v$, the Lov\'asz Local Lemma guarantees with positive probability that none of the events $T_v$ occur, i.e.~there is an independent transversal, as desired.
\end{proof}

\section{Constructions}\label{sec:constructions}

\begin{proof}[Proof of Proposition~\ref{prop:symmetricasymmetric}]
Let $\Delta=2^{2^k}$ for some integer $k>1$ and let $G$ be the complete bipartite graph with $|A_G|=|B_G|=\Delta$.
	Without loss of generality, we may assume that $L$ induces an arbitrary disjoint collection of $\Delta$ sets of size $\Delta/\log_4\Delta$ (for $A_H$) and $\Delta$ sets of size $2$ (for $B_H$). We next describe how to define $H$ with respect to $L$.
We write $A_G=\{v_1,\dots,v_\Delta\}$.
Arbitrarily partition the vertices of $B_G$ into $\Delta/\log_2\Delta$ parts of size $\log_2\Delta$, call them $C_1,\dots,C_{\Delta/\log_2\Delta}$.
Similarly, for each $v_j\in A_G$, arbitrarily partition $L(v_j)$ into $\frac12\Delta/\log_4\Delta=\Delta/\log_2\Delta$ pairs, write them $(p^{v_j}_1,\overline{p}^{v_j}_1),\dots,(p^{v_j}_{\Delta/\log_2\Delta},\overline{p}^{v_j}_{\Delta/\log_2\Delta})$.
Fix $i \in \{1,\dots,\Delta/\log_2\Delta\}$.
Note that the sub-cover of $G$ with respect to $L$ induced by $C_i$ has exactly $2^{\log_2\Delta}=\Delta$ possible independent transversals, since it has $\log_2\Delta$ parts of size $2$. 
Call these transversals $T^i_1,\dots,T^i_\Delta$. For each $j\in\{1,\dots,\Delta\}$, add a union of perfect matchings between each pair of $C_i$ and the pair $(p^{v_j}_i,\overline{p}^{v_j}_i)$ according to the transversal $T^i_j$ of $C_i$ as follows. Connect the chosen vertex of each pair with vertex $p^{v_j}_i$ and the non-chosen vertex of each pair with vertex $\overline{p}^{v_j}_i$. Note this does not violate the maximum degree condition.

Now consider any transversal $T$ of the sub-cover of $G$ with respect to $L$ induced by $B_H$. This corresponds, say, to sub-transversals $T^1_{j_1}$, \dots, $T^{\Delta/\log_2\Delta}_{j_{\Delta/\log_2\Delta}}$ of $C_1$, \dots, $C_{\Delta/\log_2\Delta}$, respectively.
Note for each $i \in \{1,\dots,\Delta/\log_2\Delta\}$ that the construction ensures, among $(p^{v_1}_i,\overline{p}^{v_1}_i),\dots,(p^{v_\Delta}_i,\overline{p}^{v_\Delta}_i)$, that only $(p^{v_{j_i}}_i,\overline{p}^{v_{j_i}}_i)$ and the pair corresponding to the transversal of $C_i$ complementary to $T^i_{j_i}$ have some vertex that has no neighbour in the transversal $T$. It follows that there are at most $2\Delta/\log_2\Delta$ vertices in $A_H$ that have no neighbour in $T$.
However, we need $\Delta$ such vertices in order to be able to extend $T$ to an independent transversal of $H$ with respect to $L$. Noting that $k>1$ implies $\Delta > 2\Delta/\log_2\Delta$, this completes the proof.
\end{proof}

\begin{proof}[Proof of Proposition~\ref{prop:asymmetric}]
We define $G$ and $H$ as follows. 
Let $|A_G|=k^k$ and $|B_G|=k$ (so that $|A_H|=k^{k+1}$ and $|B_H|=k^2$).
From each possible $k$-tuple of vertices taken from $\prod_{w\in B_G}L(w)$, we add an (arbitrary) matching to the $k$ vertices of $L(v)$ for some distinct $v\in A_G$. This satisfies the degree requirements and any transversal of $B_H$ cannot be extended to an independent transversal of $H$ with respect to $L$, as required.
\end{proof}

\begin{proposition}\label{prop:classicDP}
	For any $k \ge 2$, consider a complete bipartite graph $G=(V=A\cup B,E)$ with $|B|=k$.
	If $|A| < k^k/k!$, then for any bipartite correspondence-cover $H$ of $G$ with respect to some $L$ such that $|L(v)| \ge k$ for all $v\in A\cup B$, $H$ admits an independent transversal with respect to $L$.
	If $|A| \ge  \frac{k^{k+1}}{ k! } \log k $, then there exists a bipartite correspondence-cover $H$ of $G$ with respect to some $L$ such that $|L(v)| = k$ for all $v\in A\cup B$ and such that $H$ does not admit an independent transversal with respect to $L$.
\end{proposition}

\begin{proof}
	First assume $|A| < k^k/k!$.
	Let $H$ be any bipartite correspondence-cover of $G$ with $|L(v)|=k$ for every $v \in A\cup B.$ Note that by restricting to some subgraph $H'$ if necessary we can assume this.
	For every $v \in A$, there are at most $k!$ transversals $T$ of $L(B)$ such that every vertex in $L(v)$ has a neighbour in $T$.
	Since there are $k^k$ choices for transversals of $L(B)$ and $|A|k! < k^k$, there exists a transversal of $L(B)$ that can be extended to an independent transversal of $H$ with respect to $L$.

	Now let $|A| > \frac{k^{k+1}}{ k! } \log k $. 
	We will construct a correspondence-cover $H$ of $G$ with respect to some $L$ such that $|L(v)|=k$ for every $v \in A \cup B$ and such that $H$ admits no independent transversal with respect to $L$.
	Without loss of generality, we may assume that $L$ induces an arbitrary collection of $|A|+|B|$ disjoint sets of size $k$.
	To specify $H$ with respect to $L$, let us consider a random bipartite correspondence-cover of $G$ formed by taking a uniformly random perfect matching between $L(v)$ and $L(w)$ for each $v\in A$ and $w\in B$. Now, for each transversal $T$ of $L(B)$, the probability that every vertex $v\in A$ has at least one element $c\in L(v)$ without a neighbour in $T$ is exactly $(1-k!/k^k)^{|A|}$. Thus the expected number of transversals of $L(B)$ that can be extended to an independent transversal is $k^k(1-k!/k^k)^{|A|}$, which is less than $1$ since $|A| > \frac{k^{k+1}}{ k! } \log k $. The existence of the promised $H$ is guaranteed by the probabilistic method.
\end{proof}

The strengthened form of Corollary~\ref{cor:symmetricasymmetric} under only a maximum degree condition on the cover graph $H$ fails. In the following proposition we prove that an optimal choice for $k_A$ in this case will not even be linear in $D$.

\begin{proposition}\label{prop:symmetricasymmetricD}
	For all $D$, there exist bipartite graphs $G$ and $H$ with bipartitions $(A_G,B_G)$, $(A_H,B_H)$, respectively, 
	such that $H$ is a bipartite correspondence-cover of $G$ with respect to some $L: A_G\to 2^{A_H}, B_G\to 2^{B_H}$ and such that the following holds.
	The maximum degree of $H$ is at most $D$, $|L(v)| = \Omega(D^{8/5}) $ for all $v\in A_G$, $|L(w)|=2$ for all $w\in B_G$, and	$H$ does not admit an independent transversal with respect to $L$.
\end{proposition}

\begin{proof}
	Let $q$ be a power of $16$ which is sufficiently large ($q \ge 256$ suffices) and let $D=\left(q^{1/4}+1\right)(q+1)$. Note that with suitable rounding the following argument also works for $q$ a prime power with exponent divisible by $4$. 
	Although this will only prove the statement for certain values of $D$, the reader should be able to routinely check that it holds for all $D$, since the primes are sufficiently dense (Bertrand's postulate is sufficient for this, but one also can simply take the largest value of $k$ for which $q=16^k$ satisfies $\left(q^{1/4}+1\right)(q+1) \le D$).
	
	Let $|A_G|=2a:=2\left(q^{1/4}+1\right)(q+1)q^{1/4}$, and write $A_G=\{v_1,v_2, \ldots, v_{2a}\}.$
 	Let $k=q^2+q+1.$ Note that $k=\Omega(D^{8/5}) $ as $q$ and hence $D$ goes to infinity.
 	For every $1 \le i \le 2a$, let $L(v_i)=\{x_{i,1},\ldots, x_{i,k}\}$.
	
	So far we have only defined $A_G$, $A_H$, and $L$ so that $|L(v)| = k$ for all $v\in A_G$. 
	(So it is trivially a bipartite correspondence-cover at this point, taking $B_G=B_H=E(H)=\emptyset$.)
	We will further define $B_G$, $B_H$, and $L$ in successive stages while maintaining that $|L(w)| = 2$ for all $w\in B_G$. Throughout these stages, we will also specify the edges of $H$ while maintaining that $H$ is a bipartite correspondence-cover of $G$ with respect to $L$ and that $H$ has maximum degree at most $D$.
	At the end we show that $H$ admits no independent transversal with respect to $L$.

 	For every $1 \le j \le k$, let $X_j$ be the set of all vertices $x_{i,j}$, where $1 \le i \le a$ and let $X'_j$ be the set of all vertices $x_{i,j}$, where $a+1 \le i \le 2a$.
	We first prove the following claim, which also holds analogously with $X'_j$ instead of $X_j$.
	\begin{claim}
		For each $1\le j\le k$, one can add edges between the vertices in $X_j$ and some of the vertices in $L(w)$ for some additional vertices $w \in B_G$ with $|L(w)|=2$ so that the degree in $H$ of every $x_{i,j}$, $1\le i\le a$, is increased by at most $(q^{1/4}+1) \log_2{D}$, the degree of every additional vertex in $H$ is at most $D$, $H$ remains a bipartite correspondence-cover of $G$ with respect to $L$, and no independent transversal of $H$ with respect to $L$ may contain two vertices from $X_j$.
	\end{claim}
\begin{claimproof}
	Let $r=q^{1/4}.$ Divide the $a$ vertices of $X_j$ into $r^2+r+1$ parts of nearly equal size (being at most $q+1$). By considering a projective plane of order $r$, i.e.~a $(r^2+r+1,r+1,1)$-design, we can form $r^2+r+1$ different unions of $r+1$ parts each. Such a union contains at most $D$ elements. For each such union, we take $\ell=\log_2 D$ new vertices $w_1^j, \ldots ,w_{\ell}^j$ in $B_G$ with $|L( w_1^j)| = |L(w_2^j)| = \cdots = |L(w_{\ell}^j)|=2$, and match each vertex in the union with a distinct transversal of $L(\{w_1^j,w_2^j,\dots,w_{\ell}^j\})$, joining edges across. 
	Note that all of the additional vertices in $B_H$ have degree at most $D$ and the degree of any $x_{i,j}$, $1\le i\le a$, has increased by  $(r+1)\ell=(q^{1/4}+1)\log_2 D$.
	Moreover, we have not added any edges to $H$ that would violate the bipartite correspondence-cover condition.
	By the definition of the design, every two vertices  $x_{i_1,j}$ and  $x_{i_2,j}$, $1\le i_1,i_2\le a$, belong to a common union and hence are joined to two distinct transversals of some $L(\{w_1^j,w_2^j,\dots,w_{\ell}^j\})$, from which the conclusion follows.	
\end{claimproof}

Let us invoke this first claim for every possible $j$, both for the $X_j$'s and the $X'_j$'s.
So now we know that every $X_j$ and  $X'_j$ may have at most one element from an independent transversal of $H$.
We also have that the degree of any vertex in $A_H$ is $(q^{1/4}+1)\log_2 D$.
Next, for specified $1\le j, j'\le k$, we show how to augment the construction in such a way that the vertices in $X_j$ and $X'_{j'}$ gain additional degree of at most $q^{1/4}$ and no independent transversal in $H$ can contain vertices from both $X_j$ and $X'_{j'}$.

\begin{claim}\label{clm:no2blocksused}
Given $1\le j, j'\le k$, one can add edges between the vertices in $X_j \cup X'_{j'}$ and in $L(w)$ for some new vertices $w \in B_G$ with $|L(w)|=2$ so that the vertices in $X_j \cup X'_{j'}$ gain additional degree in $H$ of at most $q^{1/4}$, the degree of every additional vertex in $H$ is at most $D$, $H$ remains a bipartite correspondence-cover of $G$ with respect to $L$, and no independent transversal of $H$ with respect to $L$ may contain both a vertex from $X_j$ and from $X'_{j'}$.
\end{claim}
\begin{claimproof}
Let $r=q^{1/4}.$
Partition the vertices of $X_j$ into $r$ parts $S_1,\dots,S_r$ of size $D$ and similarly $X'_{j'}$ into $r$ parts $S'_1,\dots,S'_r$ of size $D$.
For every pair of parts $S_s$ and $S'_{s'}$, we add a new vertex $w$ in $B_G$ with $|L(v)|=2$, joining all vertices in $S_s$ to one of the vertices in $L(w)$ and joining all vertices in $S'_{s'}$ to the other.
Note that every additional vertex in $B_H$ has degree $D$
and every vertex in $X_j\cup X'_{j'}$ has been joined to $r$ additional vertices in $B_H$.
	Moreover, we have not added any edges to $H$ that would violate the bipartite correspondence-cover condition.
 The conclusion follows from the fact that any pair of a vertex in $X_j$ and a vertex in $X'_{j'}$ are joined to two different vertices in $L(w)$ for some $w \in B_G$. 
\end{claimproof}

Consider a $(k=q^2+q+1,q+1,1)$-design on $[k]$ with blocks $B_1,B_2, \ldots, B_k$.
For every $j \in [k]$, let us apply this second claim between $X_j$ and $X'_{j'}$ for each $j' \in B_j$.
After this, by the definition of the design as well as the degree promises of the claims, the degree of any vertex in $A_H$ is at most $(q+1)q^{1/4}+(q^{1/4}+1)\log_2{D}<(q+1)q^{1/4} +(q+1)=D$ for $q$ sufficiently large. The claims have also allowed us to maintain the other desired properties for $H$.

Suppose now, for a contradiction, that there is an independent transversal $T$ of $H$ with respect to $L$.
Since $T$ must contain exactly $2a$ vertices from $A_H$, it follows from the first claim that $T$ should contain exactly one vertex from each of $a$ distinct $X_j$ and $a$ distinct $X'_{j'}$. As such, let us suggestively write $T=\{x_{j_1},\dots,x_{j_a},x'_{j'_1},\dots,x'_{j'_a}\}$. Since $T$ is an independent transversal, we may assume from our application of the second claim that $j'_1,\dots,j'_a$ and $B_{j_1},\dots,B_{j_a}$ induce a nonincident set of $a$ points and $a$ blocks in the $(q^2+q+1,q+1,1)$-design. However, since $a=\left(q^{1/4}+1\right)(q+1)q^{1/4}>1+(q+1)(\sqrt q -1)$, this contradicts a known extremal result on nonincidents set in such a design (see~\cite[Thm.~3.3]{Stinson11} and~\cite[Thm.~3]{DSV12}). This completes the proof.
\end{proof}

\section{An asymmetric version of Haxell's theorem}\label{sec:asymHax}

\begin{proof}[Proof of Proposition~\ref{prop:asymHax}]
	We construct an auxiliary graph $H'$ on the vertex set $A_H$, partitioned by $L\mid_{A_G}: A_G\to 2^{A_H}$.
	In $H'$, two vertices $u,v \in A_H$ are connected if and only if there exists $w \in B_G$ such that $u$ is adjacent to one of the vertices in $L(w)$ and $v$ is adjacent to the other vertex in $L(w)$.
	Note that the maximum degree of $H'$ is at most $D^2$.
	Thus by Haxell's theorem~\cite[Thm.~2]{Hax01}, $H'$ admits an independent transversal $T_A$ with respect to $L\mid_{A_G}$. Trivially $T_A$ is partial independent transversal of $H$ with respect to $L$: we next show how $T_A$ can be extended to a full independent transversal of $H$ by specifying the choices on $B_G$. Let $w\in B_G$ and write $L(w)=\{x,y\}$. If $x$ has no neighbour in $T_A$, then we may add $x$ to the independent transversal. On the other hand, if $x$ has a neighbour in $T_A$, then by the definition of $H'$ and $T_A$, it must be that $y$ has no neighbour in $T_A$, in which case we may add $y$ to the independent transversal. (Note here we have used the condition that no vertex in $A_H$ is adjacent to both vertices in $L(w)$ for some $w\in B_G$.) By doing this for all $w\in B_G$, this completes the independent transversal of $H$ with respect to $L$ and thus the proof.

For the sharpness construction, we let $A_G=\{v,v'\}$ and let $L(v)=\{x_1,\dots,x_{D^2}\}$ and $L(v')=\{x'_1,\dots,x'_{D^2}\}$. We also define $B_G=\{w_{i,j} \colon 1\le i,j\le D\}$ and let $L(w_{i,j}) = \{x_{i,j},x'_{i,j}\}$ for each $1\le i,j \le D$. To define $H$, we add edges between $x_{i,j}$ and each of $x_{(i-1)D+1},x_{(i-1)D+2},\dots,x_{iD}$ and between $x'_{i,j}$ and each of $x'_{(j-1)D+1},x'_{(j-1)D+2},\dots,x'_{jD}$ for each $1\le i,j \le D$. Note that $H$ is a $D$-regular graph, $L$ satisfies the desired part size requirements, and no vertex in $A_H$ is adjacent to both vertices in $L(w)$ for some $w\in B_G$. Suppose to the contrary that $H$ admits an independent transversal $T$ with respect to $L$. By symmetry, we may assume without loss of generality that $x_1$ belongs to $T$. By the definition of $H$, this forces that $x'_{1,j}$ for each $1\le j\le D$ must also belong to $T$, which then contradicts there being a choice from $L(v')$ for $T$.
\end{proof}

 It is worth remarking that the condition in Proposition~\ref{prop:asymHax} that no vertex in $A_H$ is adjacent to both vertices in $L(w)$ for some $w\in B_G$ is necessary. For consider the following easy star construction. Let $v\in A_G$ such that $|L(v)|= k_A$ for any $k_A>0$. In $H$ join every $x\in L(v)$ with both vertices of some $L(w)$, $w \in B$, with $|L(w)|=2$. Then $H$ has maximum degree $2$ and clearly there can be no independent transversal of $H$ with respect to $L$.

\subsection*{Acknowledgement}

We are grateful to Noga Alon for stimulating discussions. We thank the anonymous referees for their careful reading that led to improvements in the presentation of this work. We especially appreciate one of the referees, for a comment that triggered our subsequent investigations around Propositions~\ref{prop:asymHax} and~\ref{prop:symmetricasymmetricD}.

\bibliographystyle{abbrv}
\bibliography{bitrans}

\end{document}